\documentclass[pdf]{article}

\usepackage{pdfpages}
\usepackage{hyperref}
\usepackage{xcolor}
\usepackage[utf8]{inputenc}
\usepackage{amsmath}
\usepackage{amsfonts}
\usepackage{amssymb}
\usepackage{amsthm}
\usepackage{marginnote}
\usepackage{enumerate}
\usepackage{graphicx}
\usepackage[normalem]{ulem}
\usepackage{fancyhdr}

\newcommand{\defstyle}[1]{\textbf{#1}}
\newcommand{\myprob}[1]{\mathbb P \left[ #1 \right]}

\newcommand{\omid}[1]{\mathbb E \left[ #1 \right]}

\newcommand{\omidCond}[2]{\mathbb E \left[ #1 \left| #2 \right. \right]}

\newcommand{\norm}[1]{\left| #1 \right|}

\newcommand{\identity}[1]{1_{#1}}
\newcommand{\bs}[1]{\boldsymbol{#1}}

\newcommand{\del}[1]{}
\newcommand{\unwritten}[1]{}

\newcommand{\restrict}[2]{{
		\left.\kern-\nulldelimiterspace 
		#1 
		\vphantom{\big|} 
		\right|_{#2} 
}}

\theoremstyle{theorem}
\newtheorem{theorem}{Theorem}[section]
\newtheorem{lemma}[theorem]{Lemma}

\theoremstyle{definition}
\newtheorem{definition}[theorem]{Definition}

\theoremstyle{definition}
\newtheorem{remark}[theorem]{Remark}

\newtheorem{algorithm}[theorem]{Algorithm}

\theoremstyle{theorem}

\numberwithin{equation}{section}

\makeatletter
\let\orgdescriptionlabel\descriptionlabel
\renewcommand*{\descriptionlabel}[1]{%
	\let\orglabel\label
	\let\label\@gobble
	\phantomsection
	\edef\@currentlabel{#1}%
	\let\label\orglabel
	\orgdescriptionlabel{#1}%
}

\makeatother

\begin{document}

\title{An Improved Lower Bound on the Largest Common Subtree of Random Leaf-Labeled Binary Trees}

\author{Ali Khezeli \footnote{Inria Paris, ali.khezeli@inria.fr} \footnote{Department of Applied Mathematics, Faculty of Mathematical Sciences, Tarbiat Modares University, P.O. Box 14115-134, Tehran, Iran, khezeli@modares.ac.ir}}

\maketitle

\begin{abstract}
	It is known that the size of the largest common subtree (i.e., the maximum agreement subtree) of two independent random binary trees with $n$ given labeled leaves is of order between $n^{0.366}$ and $n^{1/2}$. We improve the lower bound to order $n^{0.4464}$ by constructing a common subtree recursively and by proving a lower bound for its asymptotic growth. The construction is a modification of an algorithm proposed by D. Aldous by splitting the tree at the centroid and by proceeding recursively.
	
	\textbf{Keywords.} random tree, maximum agreement subtree.
\end{abstract}


\section{Introduction}


In some models of random combinatorial structures, one might also be interested in the largest common substructure of two independent instances. This is a measure of similarity of the two instances. A highly investigated example is the largest increasing subsequence of a random permutation, which can be regarded as the largest common substructure of two random total orders (see~\cite{aldous} for a discussion of this problem).
In this paper, we study the largest common subtree of two random binary trees. 
The trees considered here are \textit{leaf-labeled binary trees of size $n$}, which are binary trees with $n$ leaves together with a  numbering of the leaves with distinct numbers $1,\ldots,n$. 
By considering the trees up to label-preserving isomorphisms, one can define the notion of a \textit{uniform random leaf-labeled binary tree} with $n$ leaves. 
Let $\bs T$ and $\bs T'$ be two independent uniform random trees as above. Let $\kappa(\bs T, \bs T')$ be the size of the largest common subtree, which is, the largest subset $A$ of $\{1,\ldots,n\}$ such that $\bs T$ and $\bs T'$ induce the same binary trees with leaf set $A$. See Section~\ref{sec:definitions} for more formal definitions.

It was shown in~\cite{Pi21} that $\omid{\kappa(\bs T, \bs T')}$ is of order at most $n^{1/2}$ using the first moment method. For the lower bound, \cite{BeSt15} showed that $\omid{\kappa(\bs T, \bs T')}$ is of order at least $n^{1/8}$. This bound was improved by D. Aldous in~\cite{aldous} to $n^{\beta}$, where $\beta=(\sqrt 3-1)/2\sim 0.366$. In this paper, we improve the lower bound to order $n^{0.4464}$ (improving the upper bound remains an open problem; see Open Problem~2 of~\cite{aldous}). More precisely,

\begin{theorem}
	\label{thm:main}
	For the constant $0.4464<\beta<0.4465$ given by~\eqref{eq:beta} below and for every $\beta'<\beta$, one has $\omid{\kappa(\bs T, \bs T')}=\Omega(n^{\beta'})$.
\end{theorem}

The proof uses an algorithm to construct a common subtree recursively.
We will also prove the same result for rooted and doubly-rooted trees (Theorem~\ref{thm:induction}).

Some other variants of the model have also been investigated in the literature. Originally, the concept of the largest common subtree was introduced in~\cite{FiGo85} (under the name \textit{maximum agreement subtree} or MAST) to measure the similarity of phylogenetic trees. In the worst case, the MAST of two leaf-labeled rooted binary trees on $n$ leaves is of order $\Theta(\sqrt{\log n})$ (\cite{Ma20}). For \textit{balanced} labeled binary trees, the lower bound is improved to order $n^{1/6}$ in~\cite{deterministictrees} (not that by disregarding the labels, there is a single balanced binary tree of a given size). 
In the random case, \cite{MiSu19} proves the exact rate of $n^{1/2}$ for the expected size of MAST of a random pair of rooted binary trees that have a common random shape (not necessarily balanced). In this model, the first tree is a random rooted binary tree and the second tree is obtained by a random permutation of the leaves.


The structure of this paper is as follows. In the rest of the introduction, an outline of the algorithm and a heuristic analysis are provided. Section~\ref{sec:definitions} provides the definitions and some lemmas regarding splitting the trees at their centroids. The formal algorithm and the proof of the main theorem is then given in Section~\ref{sec:proof}.

\subsection{Outline of the Algorithm}

The proof of Theorem~\ref{thm:main} is by improving the algorithm of~\cite{aldous} to construct a common subtree recursively. The general structure of the algorithm of~\cite{aldous} can be summarized as follows. 
\begin{algorithm}[Aldous, 2022]
	\label{alg:general}
	Given two binary trees $t$ and $t'$ on the same leaf set,
	\begin{enumerate}[(i)]
		\item \label{alg:general:branch}
		Choose branch points $b$ and $b'$ in $t$ and $t'$ respectively and match $b$ to $b'$.
		\item \label{alg:general:sort} By splitting at $b$, $t$ is decomposed into three subtrees (see Figure~\ref{fig:splitting}). Choose a \textit{suitable} order for the subtrees and denote their leaf sets by $A_1,A_2$ and $A_3$. By doing the same for $t'$, the leaves are partitioned into $A'_1, A'_2$ and $A'_3$. 
		\item \label{alg:general:recursion} For each $i$, two binary trees are induced on the leaf set $A_i\cap A'_i$. Recursively, construct a common subtree of these two trees that contains the distinguished leaf corresponding to $b$ (and any other distinguished leaf from the previous stages). Join the three resulting subtrees to obtain a common subtree of $t$ and $t'$.
	\end{enumerate}
\end{algorithm}

There are some choices to be made in this general sketch. In~\cite{aldous}, the branch point $b$ in the first stage is selected by choosing three random leaves and considering the branch point separating them. In the next stages, all trees have two distinguished leaves and the branch point is chosen by selecting a third leaf randomly. The crucial facts that allow the analysis of the algorithm are: (1) the distribution of the sizes of the three branches converges to a known distribution as the number of leaves grows, (2) conditioned on the sizes, the distribution of the three leaf sets $(A_1,A_2,A_3)$ is uniform and (3) conditioned on the leaf sets, the three subtrees of $t$ are uniform and independent. Using these facts, a heuristic identity is obtained in~\cite{aldous} for the asymptotic growth (see Subsection~\ref{subsec:heuristic} below). To obtain a formal proof, \cite{aldous} controls the convergence and uses a martingale argument. 

Facts (2) and (3) above hold for other choices in the algorithm as well, as long as the probability of choosing a branch point in step~\eqref{alg:general:branch} is a function of the sizes of the three branches. It is suggested in~\cite{aldous} to choose the \textit{centroid} of the trees in step~\eqref{alg:general:branch} of the algorithm and to sort the branches by their sizes in step~\eqref{alg:general:sort}, since this would result in throwing out less leaves in step~\eqref{alg:general:recursion}.
However, there are some issues in continuing the recursion (as partly mentioned in Subsection~4.2 of~\cite{aldous}), which are discussed below. 
The first issue is that sorting by size is only possible in the first stage of the recursion. In the second stage, when one splits the subtree with leaf set $A_i\cap A'_i$ into three pieces, one of the pieces contains $b$ and it must be matched to the piece containing $b'$ in the other tree. Also, this piece has two distinguished leaves, namely $b$ and $b_1$ (which are branch points in $t$). This causes another constraint on the third stage of the recursion. In addition, the next branch point should separate $b$ and $b_1$, and hence, the centroid cannot necessarily be chosen. Here, we propose to select the \textit{semi-centroid} of the path connecting $b$ to $b_1$, which will be defined similarly to the centroid. Using this, we obtain a full recursive algorithm involving three types of trees: Those with 0, 1 or 2 distinguished leaves, which are called \textit{non-rooted}, \textit{rooted} and \textit{doubly-rooted} here (non-rooted trees appear only in the first stage). The formal algorithm will be presented in Subsection~\ref{subsec:alg}. In this algorithm, despite what is mentioned in Subsection~4.2 of~\cite{aldous}, the number of constraints does not increase in the next stages. This allows us to analyze the algorithm heuristically (discussed in Subsection~\ref{subsec:heuristic}), which results in a system of equations. The exponent $\beta$ is obtained by solving these equations numerically. For a rigorous proof, we simplify the proof of~\cite{aldous} significantly 
and extend it to the new algorithm. In particular, no martingale argument is needed and the proof is by a simple induction on $n$.

The value of $\beta$ found in this method is less than the upper bound $0.485\ldots$ given in Section~4.2 of~\cite{aldous} for two reasons: One is that sorting by size is not always possible and the other is that the semi-centroid is sometimes used instead of the true centroid (which is different from the centroid with an asymptotically non-vanishing probability). However, by the above arguments, we think that this is the best bound one can expect from using the centroid if one desires that the trees in each stage are uniform and independent. 
In addition, Aldous remarks that the bound given by the centroid method is expected to be less than (but close to) the true growth exponent of $\omid{\kappa(\bs T, \bs T')}$. The reason is that, since $\omid{\kappa(\bs T, \bs T')}$ is expected to converge to a nontrivial random variable after normalization (Conjecture~6 of~\cite{aldous}), there is no guarantee that sorting by size in step~\eqref{alg:general:sort} of Algorithm~\ref{alg:general} gives the optimal rate.\footnote{Personal communication.} For the same reason, choosing the centroids in step~\eqref{alg:general:branch} might not be optimal as well.

\subsection{Heuristic analysis of the algorithm}
\label{subsec:heuristic}
For $i=0,1,2$, let $\bs T_i$ and $\bs T'_i$ be independent random binary trees on $n$ leaves, plus $i$ additional distinguished leaves. The above algorithm of splitting at the centroid constructs a common subtree of $\bs T_i$ and $\bs T'_i$ that contains the distinguished leaves. One might expect that the size of the output of the algorithm is asymptotically $\bs Z_i n^{\beta_i}$ for some constant $\beta_i$ and some random variable $\bs Z_i$ (see Conjecture~6 of~\cite{aldous}, which is for the case $i=0$). In particular, one might expect that the expected size of the output is asymptotically $C_i n^{\beta_i}$. However, there is no reason to expect $C_0,C_1$ and $C_2$ to be equal.

We follow the heuristic of~\cite{aldous} for splitting $\bs T_0$ at the centroid. The sizes of the three branches are asymptotically $(nX_1, nX_2, nX_3)$, where $(X_1,X_2,X_3)$ has a Dirichlet$(-1/2,-1/2,-1/2)$ distribution (see Definition~\ref{def:dirichlet}) on the subset of the simplex $\{(x_1,x_2,x_3)\in(\mathbb R^+)^3: \sum_i x_i=1 \}$ defined by the condition $X_1\leq X_2\leq X_3\leq \frac 12$. Similarly, the sizes of the three branches of $\bs T'_0$ are asymptotically $(nX'_1,nX'_2,nX'_3)$, independently of $\bs T_0$. Conditioned on knowing the sizes of the branches, the sizes of $A_i\cap A'_i$, $i=1,2,3$ are very concentrated around their means $(nX_1X'_1, nX_2X'_2, nX_3X'_3)$. Therefore, the recursion gives heuristically that
\begin{equation*}
	C_0 n^{\beta_0} = \omid{C_1(nX_1X'_1)^{\beta_1} + C_1(nX_2X'_2)^{\beta_1} + C_1(nX_3X'_3)^{\beta_1}}.
\end{equation*}
This gives that $\beta_0=\beta_1=:\beta$ and 
\begin{equation}
	\label{eq:rec0}
	C_0 = C_1 \omid{X_1^\beta}^2 + C_1 \omid{X_2^\beta}^2 + C_1 \omid{X_3^\beta}^2.
\end{equation}
This is different from Equation~(2) of~\cite{aldous} since the right hand side involves $C_1$, which corresponds to having a distinguished leaf after splitting (see Figure~\ref{fig:splitting}). 

\begin{figure}[t]
	\begin{center}
		\includegraphics[width=.9\textwidth]{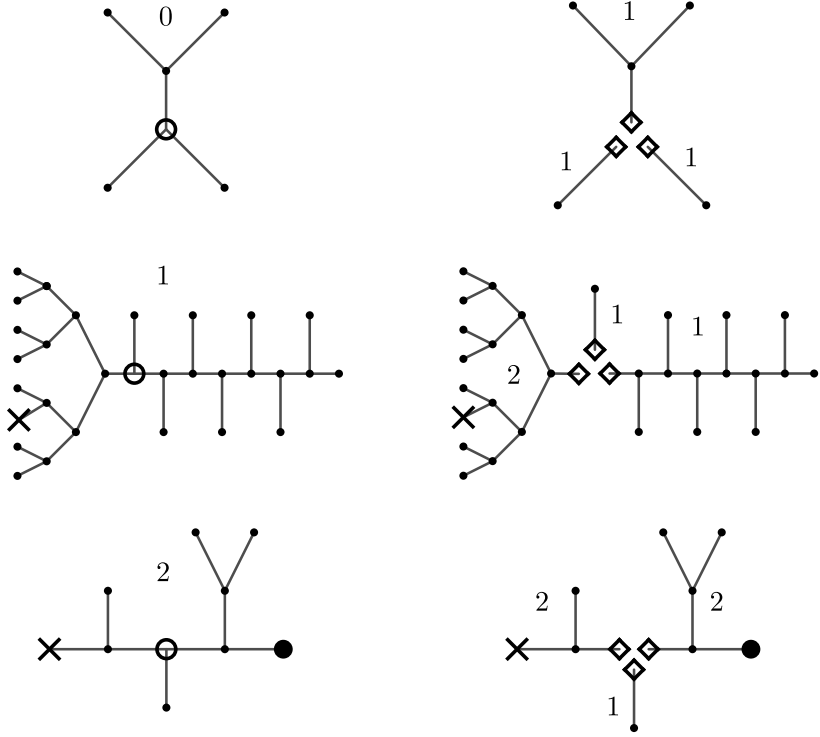}
	\end{center}
	\caption{Splitting the trees. In the left, three trees are shown. The distinguished leaves $\star$ and $\bullet$ are depicted by $\times$ and $\bullet$ respectively. The number of distinguished leaves of each tree is shown by a number. A branch point (which is the centroid or semi-centroid) is selected and shown by $\circ$. In the right, the branch point is split into three new leaves, which are now considered as distinguished leaves and shown by $\diamond$.}
	\label{fig:splitting}
\end{figure}

We can do similar calculations for splitting $\bs T_1$ at the centroid. We will show that the sizes of the branches are asymptotically $(nU_1,nU_2,nU_3)$, where $(U_1,U_2,U_3)$ has a Dirichlet$(-1/2,-1/2,1/2)$ distribution on the subset of the simplex corresponding to the conditions $U_1\leq  U_2\leq \frac 12$ and $U_3\leq \frac 12$. Here, $U_3$ corresponds to the branch containing the distinguished leaf and we have sorted the other two branches. By a similar heuristic argument, one obtains that $\beta_2=\beta$ and
\begin{equation}
	\label{eq:rec1}
	C_1 = C_1 \omid{U_1^\beta}^2 + C_1 \omid{U_2^\beta}^2 + C_2 \omid{U_3^\beta}^2.
\end{equation}
In this equation, the coefficient $C_2$ corresponds to the branch with two distinguished leaves (in fact, the heuristic gives $\beta_2\leq\beta$, and the inequality $\beta_2\geq \beta$ follows from splitting $\bs T_2$ discussed below). 

Finally, when splitting $\bs T_2$ at the semi-centroid of the two distinguished leaves, the sizes of the three branches are asymptotically $(nV_1,nV_2,nV_3)$, where $(V_1,V_2,V_3)$ has a Dirichlet$(1/2,1/2,-1/2)$ distribution on the subset of the simplex corresponding to the conditions $V_1\leq\frac 12$ and $V_2\leq \frac 12$. Here, $V_1$ and $V_2$ correspond to the branches containing the two distinguished leaves, sorting is not possible and there is no restriction on $V_3$. One finds similarly
\begin{equation}
	\label{eq:rec2}
	C_2 = C_2 \omid{V_1^\beta}^2 + C_2 \omid{V_2^\beta}^2 + C_1 \omid{V_3^\beta}^2.
\end{equation}
In this equation, the coefficient $C_1$ corresponds to the third branch which has only one distinguished leaf. By letting $\alpha:=C_2/C_1$, one can solve $\alpha$ in terms of $\beta$ in either of~\eqref{eq:rec1} and~\eqref{eq:rec2}:
\begin{equation}
	\label{eq:beta}
	\alpha=\frac{1-\omid{U_1^\beta}^2 - \omid{U_2^\beta}^2}{\omid{U_3^\beta}^2} = \frac{\omid{V_3^\beta}^2}{1-\omid{V_1^\beta}^2-\omid{V_2^\beta}^2}.
\end{equation}
By a monotonicity and continuity argument, there exists a unique solution for $\beta$. Simulation gives that $0.4464<\beta<0.4465$ and $\alpha\approx 0.859$. Although the above arguments are heuristic and don't contain a formal proof, \eqref{eq:beta} is very important for being able to use induction in the proof of the main theorem.

\section{Random Binary Trees}
\label{sec:definitions}

\subsection{The Model}
\label{subsec:model}
A binary tree $T$ is a finite tree in which every vertex has degree 1 or 3. The vertices of degree 1 are called the leaves and the other vertices are called the branch points. 
If $B$ is a subset of leaves, then $T$ induces a binary tree with leaf set $B$ naturally, which is denoted by $\restrict{T}{B}$ (first consider the subtree spanned by $B$ and then remove the vertices of degree two).
Call two binary trees with a common leaf set $A$ equivalent if there exists a graph isomorphism between them such that its restriction on $A$ is the identity; i.e., preserves the labels of the leaves. Under this equivalence relation, there are finitely many classes which are called \defstyle{non-rooted labeled binary trees} with leaf set $A$. 
Here, by a \defstyle{random binary tree} with leaf set $A$, we mean a random element of this finite set chosen uniformly.

Fix $n\geq 0$.
Define the random binary trees $\bs T_0$ (if $n\neq 0$), $\bs T_1$ and $\bs T_2$ with the uniform distribution on the set of trees with the following leaf sets:
\begin{eqnarray*}
	\bs T_0 &\text{with leaf set}& \{1,\ldots,n\},\\
	\bs T_1 &\text{with leaf set}& \{1,\ldots,n,\star\},\\
	\bs T_2 &\text{with leaf set}& \{1,\ldots,n,\star,\bullet\}.
\end{eqnarray*}
Here, $1,\ldots,n$ are called \textbf{the original leaves} and $\star$ and $\bullet$ are arbitrary elements called \textbf{the distinguished leaves}. So $\bs T_i$ has $i$ distinguished leaves. We say that $\bs T_0, \bs T_1$ and $\bs T_2$ are \defstyle{non-rooted}, \defstyle{rooted} and \defstyle{doubly-rooted} respectively.\footnote{Usually, rooted binary trees are defined such that the root has degree 2. This is in bijection with the above definition when $n>0$, since deleting the distinguished leaf results in a vertex of degree two.}
We say that the \defstyle{size} of $\bs T_i$ is the number of original leaves; i.e., $n$. 

Let $\bs T'_i$ be another random binary tree defined similarly and independently of $\bs T_i$.
If $B$ is a subset of the leaves such that $B$ contains the distinguished leaves $\star$ and $\bullet$ (if they are present) and $\restrict{(\bs T_i)}{B}$ is equivalent to $\restrict{(\bs T'_i)}{B}$, then we call the latter a \textbf{common subtree} of $\bs T_i$ and $\bs T'_i$. Define $\kappa(\bs T_i,\bs T'_i)$ to be the size of the largest common subtree of $\bs T_i$ and $\bs T'_i$.
We stress again that the distinguished leaves are contained in the subtrees but are not counted in their sizes.


We recall some of the properties of random binary trees from~\cite{aldous}.
Such trees can be constructed recursively by adding the leaves one by one and attaching them to (the middle of) one of the existing edges chosen uniformly. Since the number of edges is two times the number of leaves minus 3, one finds that the number of binary trees on $n$ given leaves is
\[
	c_n:=1\times 3\times 5 \times \cdots \times (2n-5) = \frac{(2n-5)!}{2^{n-3} (n-3)!}.
\]
Stirling's formula gives that $c_n\sim 2^{n-3/2}e^{-n}n^{n-2}$. In particular, this implies
\begin{equation}
	\label{eq:c}
	 \frac{c_{n+2}}{n!}\sim \pi^{-1/2}2^n n^{-1/2},\quad \frac{c_{n+1}}{n!}\sim \pi^{-1/2}2^{n-1} n^{-3/2}.
\end{equation}
A fundamental simple property is the \textit{consistency property} expressed in the following lemma.
\begin{lemma}
	\label{lem:consistency}
	If $\bs T$ is a uniform random binary tree with leaf set $A$ and $B\subset A$, then $\restrict{\bs T}{B}$ is a uniform binary tree with leaf set $B$. If $B$ is random and independent of $\bs T$, then the same holds when conditioned on $B$.
\end{lemma}

\subsection{Splitting at the Centroid}
\label{subsec:splitting}

As mentioned in Subsection~\ref{subsec:heuristic}, it is known that by splitting $\bs T_0$ at the centroid, the sizes of the three branches asymptotically have the Dirichlet$(-1/2,-1/2,-1/2)$ distribution on a specific subset of the simplex (see Definition~\ref{def:dirichlet} below). However, we do not need this since the recursion only involves rooted and doubly-rooted trees (the theorem for non-rooted trees will follow from a simple comparison with rooted trees; see Lemma~\ref{lem:coupling}). We analyze splitting these types of trees in this subsection. 

First, we analyze splitting $\bs T_2$. Let $b$ be a branch point on the path connecting the leaves $\star$ and $\bullet$. 
By replacing $b$ with three distinct leaves and connecting them to the neighbors of $b$, $\bs T_2$ is split into three binary trees $\bs S_1,\bs S_2$ and $\bs S_3$. Choose the order such that $\star$ and $\bullet$ belong to $\bs S_1$ and $\bs S_2$ respectively. 
Note that $S_3$ has one distinguished leaf (corresponding to $b$) and the other two trees have two distinguished leaves (corresponding to $b$ and one of $\star$ and $\bullet$).
Let $\bs l_i$ be the number of original leaves in $\bs S_i$. Assume $b$ is the \textbf{semi-centroid} of $\bs T_2$ from the beginning, which is uniquely defined by the conditions $\bs l_1< \frac n2$ and $\bs l_2\leq\frac n2$. 
See Figure~\ref{fig:splitting} for an illustration.

The above decomposition provides a natural bijection, which enables one to analyze the distribution of the three branches: Every partitioning of $\{1,\ldots,n\}$ into three sets $(A_1,A_2,A_3)$ with sizes $(l_1,l_2,l_3)$, such that $0\leq l_1< \frac n2$ and $0\leq l_2\leq \frac n2$, and every triple of trees $(S_1,S_2,S_3)$ with leaf sets $A_1\cup\{\star, b\}$, $A_2\cup\{\bullet,b\}$ and $A_3\cup\{b\}$ respectively (where $b$ is an abstract element here), determine a unique doubly-rooted binary tree. This implies the following lemma straightforwardly.

\begin{lemma}
	\label{lem:split2}
	By splitting $\bs T_2$ at the semi-centroid,
	\begin{enumerate}[(i)]
		\item \label{lem:split2:1} For every $l_1,l_2,l_3\in\mathbb N\cup\{0\}$ such that $\sum_i l_i=n$, $l_1< \frac n2$ and $l_2\leq \frac n2$,
		\[
		\myprob{\forall i: \bs l_i=l_i} = \frac{\binom{n}{l_1\  l_2\  l_3}\;c_{l_1+2}\;c_{l_2+2}\;c_{l_3+1}}{c_{n+2}}.
		\]
		\item Conditioned on $\forall i: \bs l_i=l_i$, the sets of original leaves of $\bs S_1,\bs S_2$ and $\bs S_3$ are uniformly at random among all possible partitions of $\{1,\ldots,n\}$ with the given sizes.
		\item Conditioned on the three leaf sets, $\bs S_1, \bs S_2$ and $\bs S_3$ are independent and uniform random binary trees, respectively having 2, 2 and 1 distinguished leaves.
	\end{enumerate}
\end{lemma}

One can deduce the asymptotic distribution of the three branch sizes. First, recall the following definition.
\begin{definition}
	\label{def:dirichlet}
	Let $E$ be a subset of the simplex $\{(x_1,x_2,x_3): x_1,x_2,x_3\geq 0, \sum x_i=1\}$ and $a_1,a_2,a_3\in\mathbb R$. The \defstyle{Dirichlet distribution} on $E$ with parameters $(a_1,a_2,a_3)$ is the probability measure on $E$ whose density is proportional to $\prod_i x_i^{a_i-1}$ (assuming the integral on $E$ is finite). If $a_1,a_2,a_3>0$, then the Dirichlet distribution makes sense on the whole simplex. However, negative values of $a_1,a_2,a_3$ can also be allowed if $E$ is properly chosen.
\end{definition}

Part~\eqref{lem:split2:1} of the above lemma, together with~\eqref{eq:c}, imply that

\begin{lemma}
	\label{lem:dirichlet2}
	The tuple $({\bs l_1}/{n}, {\bs l_2}/{n}, {\bs l_3}/{n})$ converges weakly to a random vector $(V_1,V_2,V_3)$ which has a Dirichlet$(1/2,1/2,-1/2)$ distribution on the subset of the simplex corresponding to $V_1\leq\frac 12$ and $V_2\leq \frac 12$. 
\end{lemma}

Now, we discuss splitting $\bs T_1$.
Let $b$ be the centroid of $\bs T_1$; i.e., a branch point such that after splitting at $b$, each of the three branches has at most half of the leaves. If there are two such vertices (which have to be neighbors and this happens only when one branch has exactly half of the leaves), choose the one farthest from the distinguished leaf $\star$. By splitting at $b$, one obtains 3 trees which are denoted by $\bs S_1,\bs S_2$ and $\bs S_3$ again (see Figure~\ref{fig:splitting}). Choose the order such that $\bs S_3$ contains $\star$. Note that $S_3$ has two distinguished leaves and the other two trees have one distinguished leaf. Let $\bs k_i$ be the number of original leaves (i.e., excluding $\star$) in $\bs S_i$. Choose the order of $\bs S_1$ and $\bs S_2$ such that $\bs k_1\leq \bs k_2$. For the rare cases in which $\bs k_1=\bs k_2$, choose one of the two orders randomly.
Similarly to the previous case, this decomposition provides an (almost-) bijection with the difference that it is a 2 to 1 map on the cases $\bs k_1=\bs k_2$. This implies the following result.

\begin{lemma}
	\label{lem:split1}
	By splitting $\bs T_1$ at the centroid,
	\begin{enumerate}[(i)]
		\item \label{lem:split1:1} For every $k_1,k_2,k_3\in\mathbb N\cup\{0\}$ such that $\sum_i k_i=n$, $k_1\leq k_2<\frac {n+1}2$ and $k_3\leq \frac {n-1}2$,
		\[
			\myprob{\forall i: \bs k_i=k_i} = \frac{\delta\binom{n}{k_1\  k_2\  k_3}\;c_{k_1+1}\;c_{k_2+1}\;c_{k_3+2}}{c_{n+1}},
		\]
		where $\delta=1$ if $k_1<k_2$ and $\delta=\frac 12$ otherwise.
		\item \label{lem:split1:2} Conditioned on $\forall i: \bs k_i=k_i$, the sets of original leaves of  $\bs S_1,\bs S_2$ and $\bs S_3$ are uniformly at random among all possible partitions of $\{1,\ldots,n\}$ with the given sizes.
		\item \label{lem:split1:3} Conditioned on the three leaf sets, $\bs S_1, \bs S_2$ and $\bs S_3$ are independent and uniform random binary trees, respectively having 1, 1 and 2 distinguished leaves.
	\end{enumerate}
\end{lemma}

One can show that the probability of $\bs k_1=\bs k_2$ converges to zero. Therefore, similarly to the previous case, part~\eqref{lem:split1:1} of the above lemma and~\eqref{eq:c} imply
\begin{lemma}
	\label{lem:dirichlet1}
	The tuple $({\bs k_1}/{n}, {\bs k_2}/{n}, {\bs k_3}/{n})$ converges weakly to a random vector $(U_1,U_2,U_3)$ which has a Dirichlet$(-1/2,-1/2,1/2)$ distribution on the subset of the simplex corresponding to $U_1\leq  U_2\leq \frac 12$ and $U_3\leq \frac 12$. 
\end{lemma}

\section{Proof of the Main Theorem}
\label{sec:proof}

\subsection{The Algorithm}
\label{subsec:alg}

As discussed in Subsection~\ref{subsec:heuristic}, the recursion deals only with rooted and doubly-rooted trees (the other case will follow by comparison; see Subsection~\ref{subsec:proof}).
Let $n\geq 0$ and let $T$ and $T'$ be binary trees on the same leaf set, both rooted or both doubly-rooted. Recursively, we construct a common subtree of $T$ and $T'$ containing the additional distinguished leaf (or leaves). For $n=0$, either both of $T$ and $T'$ are trees with a single vertex $\star$ or with a single edge $\star\bullet$. In this case, output the whole tree. For $n\geq 1$, we proceed inductively as follows.

If the trees are doubly-rooted, split them at their semi-centroids $b$ and $b'$ as described in Subsection~\ref{subsec:splitting}. Assume that splitting $T$ results in partitioning $\{1,\ldots,n\}$ into subsets $(A_1,A_2,A_3)$, where the order is specified in Subsection~\ref{subsec:splitting}. Define $(A'_1,A'_2,A'_3)$ similarly. $T$ induces binary trees $\bs R_1$ with leaf set $(A_1\cap A'_1)\cup\{\star,b\}$, $\bs R_2$ with leaf set $(A_2\cap A'_2)\cup\{\bullet,b\}$ and $\bs R_3$ with leaf set $(A_3\cap A'_3)\cup \{b\}$. Define $\bs R'_1,\bs R'_2$ and $\bs R'_3$ similarly. By relabeling $b$ and $b'$, $\bs R_3$ and $\bs R'_3$ can be regarded as rooted trees and the rest as doubly-rooted trees. The induction hypothesis defines output common subtrees for the pairs $(\bs R_i,\bs R'_i)$ for $i=1,2,3$. Join these trees by identifying the three distinguished leaves corresponding to $b$. It is clear that this forms a common subtree of $T$ and $T'$. This is the output of the algorithm.

If $T$ and $T'$ are rooted, split them at their centroids as described in Subsection~\ref{subsec:splitting}. Similarly to the previous case, redefine the subtrees $\bs R_i$ and $\bs R'_i$ for $i=1,2,3$, use the induction hypothesis to define three common subtrees (note that $\bs R_3$ and $\bs R'_3$ are doubly-rooted and the rest are rooted), glue them to find a common subtree of $T$ and $T'$, and output the resulting tree. This finishes the definition of the algorithm.

\subsection{Analysis of the Algorithm}
For $i=1,2$, let $\gamma(\bs T_i,\bs T'_i)$ be the size of the output of the algorithm of Subsection~\ref{subsec:alg} for the trees $\bs T_i$ and $\bs T'_i$ defined in Subsection~\ref{subsec:model}. Note that the dependence of these values on $n$ is not shown in the symbols for better readability. 

\begin{theorem}
	\label{thm:induction}
	Let $\alpha$ and $\beta$ be constants satisfying~\eqref{eq:beta}. 
	For every $\beta'<\beta$, there exists $\delta>0$ such that for all $n\geq 0$, 
	\begin{eqnarray}
		\label{eq:thm:induction:1}
		\omid{\gamma(\bs T_1,\bs T'_1)}&\geq& \delta n^{\beta'},\\
		\label{eq:thm:induction:2}
		\omid{\gamma(\bs T_2,\bs T'_2)}&\geq& \alpha \delta n^{\beta'}.
	\end{eqnarray}
\end{theorem}

\begin{remark}
	By similar arguments, one can prove that the rate is exactly $n^{\beta}$; i.e., for every $\beta''>\beta$, one has $\omid{\gamma(\bs T_i,\bs T'_i)}=O(n^{\beta''})$.
\end{remark}

We will prove this theorem by induction on $n$ according to the following motivation.
With the notation of the algorithm, one has 
\begin{equation}
	\label{eq:gamma}
	\gamma(\bs T_1,\bs T'_1) = \sum_i \gamma(\bs R_i,\bs R'_i).
\end{equation}
So, we need to bound $\sum_i \omid{\gamma(\bs R_i,\bs R'_i)}$. We will show that 
$\bs R_i$ is a uniform random binary tree on $A_i\cap A'_i$ (rooted or doubly-rooted as specified in the algorithm). Therefore, if the induction hypothesis holds, the conditional expectation of ${\gamma(\bs R_i,\bs R'_i)}$ is bounded from below by $\delta\norm{A_i\cap A'_i}^{\beta'}$ or $\alpha\delta \norm{A_i\cap A'_i}^{\beta'}$. The formal proof is based on this motivation and the following lemma.
\begin{lemma}
	\label{lem:induction}
	Fix $\beta'\geq 0$ and $\epsilon>0$.
	\begin{enumerate}[(i)]
		\item Let $A_i$ and $A'_i$ ($i=1,2,3$) be the leaf sets defined in the algorithm when starting from $\bs T_1$ and $\bs T'_1$. Then, for large enough $n$,
		\begin{equation}
			\label{eq:lem:induction:1}
			\forall i: \omid{\norm{A_i\cap A'_i}^{\beta'}} \geq (1-\epsilon)n^{\beta'} \omid{U_i^{\beta'}}^2,
		\end{equation}
		where $(U_1,U_2,U_3)$ is defined in Lemma~\ref{lem:dirichlet1}. 
		\item Similarly, if the sets are defined in the algorithm when starting from $\bs T_2$ and $\bs T'_2$, then for large enough $n$,
		\begin{equation}
			\label{eq:lem:induction:2}
				\forall i: \omid{\norm{A_i\cap A'_i}^{\beta'}} \geq (1-\epsilon)n^{\beta'} \omid{V_i^{\beta'}}^2,	
		\end{equation}
		where $(V_1,V_2,V_3)$ is defined in Lemma~\ref{lem:dirichlet2}. 
	\end{enumerate}
\end{lemma}

\begin{proof}
	The proof is almost the same as that of Proposition~5 of~\cite{aldous}.
	We only prove the first claim and the second one can be proved similarly. Condition on the event $E$ of $\norm{A_i}=a_i$ and $\norm{A'_i}=a'_i$ for all $i$. Fix $i\in\{1,2,3\}$. By part~\eqref{lem:split1:2} of Lemma~\ref{lem:split1}, $A_i$ and $A'_i$ are independent uniform random subsets of $\{1,\ldots,n\}$ of sizes $a_i$ and $a'_i$. So, $\norm{A_i\cap A'_i}$ has a hypergeometric distribution. Fix $\epsilon>0$ and let $F_{\epsilon}$ be the event that $a_i$ and $a'_i$ are between $\epsilon n$ and $(1-\epsilon)n$. In this regime, the hypergeometric distribution is concentrated around its mean. In particular, Equation~(16) of~\cite{aldous} gives the following: When $\epsilon$ is fixed, for large enough $n$, and for all choices of $a_i$ and $a'_i$ between $\epsilon n$ and $(1-\epsilon)n$, 
	\[
		\omidCond{\norm{A_i\cap A'_i}^{\beta'}}{E} \geq (1-\epsilon) \omidCond{\norm{A_i\cap A'_i}}{E}^{\beta'} = (1-\epsilon) \left(\frac{a_ia'_i}{n} \right)^{\beta'}.
	\]
	This implies that for large $n$,
	\[
		\omid{\norm{A_i\cap A'_i}^{\beta'}} \geq (1-\epsilon) \omid{\left(\frac{\norm{A_i}\cdot\norm{A'_i}}n\right)^{\beta'}\identity{F_{\epsilon}}}=(1-\epsilon)n^{\beta'}\omid{\left(\frac {\norm{A_i}}n\right)^{\beta'}\identity{F_{\epsilon}}}^2.
	\]
	By the convergence in Lemma~\ref{lem:dirichlet1}, one deduces that
	\[
		\liminf_n n^{-\beta'} \omid{\norm{A_i\cap A'_i}^{\beta'}} \geq (1-\epsilon)\omid{U_i^{\beta'}\identity{G_{\epsilon}}}^2,
	\]
	where $G_{\epsilon}$ is the event $\epsilon\leq U_i\leq (1-\epsilon)$. 
	Now the claim is obtained by reducing the value of $\epsilon$ suitably.
\end{proof}

We are now ready to prove Theorem~\ref{thm:induction}.

\begin{proof}[Proof of Theorem~\ref{thm:induction}]
	Fix $\beta'<\beta$. Since $\alpha$ and $\beta$ satisfy~\eqref{eq:beta} and the terms in~\eqref{eq:beta} are monotone in $\beta$ (see also~\eqref{eq:rec1} and~\eqref{eq:rec2}), there exists $\epsilon>0$ such that 
	\begin{eqnarray}
		\label{eq:epsilon1}
		\omid{U_1^{\beta'}}^2 +  \omid{U_2^{\beta'}}^2 + \alpha \omid{U_3^{\beta'}}^2 &\geq & \frac 1{1-\epsilon},\\
		\label{eq:epsilon2}
		\alpha \omid{V_1^{\beta'}}^2 + \alpha \omid{V_2^{\beta'}}^2 +  \omid{V_3^{\beta'}}^2 &\geq & \frac{\alpha}{1-\epsilon}.
	\end{eqnarray}
	For this value of $\epsilon$, by Lemma~\ref{lem:induction}, there exists $N\in\mathbb N$ such that~\eqref{eq:lem:induction:1} and~\eqref{eq:lem:induction:2} hold for all $n\geq N$. Choose $\delta>0$ such that for all $n< N$, the claims~\eqref{eq:thm:induction:1} and~\eqref{eq:thm:induction:2} of the theorem hold. 
	
	Now, we prove the claims by induction on $N$. By the definition of $\delta$, the claims hold for all $n< N$. Now, assume $n\geq N$ and both claims hold for $1,2,\ldots,n-1$. We will prove them for $n$. 
	
	First, consider the algorithm for $\bs T_1$ and $\bs T'_1$. As mentioned in~\eqref{eq:gamma}, one has $\omid{\gamma(\bs T_1,\bs T'_1)} = \sum_i \omid{\gamma(\bs R_i,\bs R'_i)}$. By part~\eqref{lem:split1:3} of Lemma~\ref{lem:split1} and the consistency property (Lemma~\ref{lem:consistency}), conditioned on knowing $A_i$ and $A'_i$, $\bs R_i$ and $\bs R'_i$ are independent random binary trees on the leaf set $A_i\cap A'_i$ (which are rooted for $i=1,2$ and doubly-rooted for $i=3$). Therefore, the induction hypothesis implies that
	\[
		\omid{\gamma(\bs T_1,\bs T'_1)} \geq \delta\omid{\norm{A_1\cap A'_1}^{\beta'}} + \delta\omid{\norm{A_2\cap A'_2}^{\beta'}} + \alpha\delta\omid{\norm{A_3\cap A'_3}^{\beta'}}. 
	\]
	Since $n\geq N$, the definition of $N$ and~\eqref{eq:lem:induction:1} give
	\[
		\omid{\gamma(\bs T_1,\bs T'_1)} \geq \delta(1-\epsilon)n^{\beta'}\left(\omid{U_1^{\beta'}}^2 +  \omid{U_2^{\beta'}}^2 + \alpha \omid{U_3^{\beta'}}^2\right). 
	\]
	By~\eqref{eq:epsilon1}, the right hand side is at least $\delta n^{\beta'}$.
	So the first induction claim~\eqref{eq:thm:induction:1} is proved. 
	
	Second, consider the algorithm for $\bs T_2$ and $\bs T'_2$. Similarly to the previous case, $\omid{\gamma(\bs T_2,\bs T'_2)} = \sum_i \omid{\gamma(\bs R_i,\bs R'_i)}$ (note that $\bs R_i, \bs R'_i, A_i$ and $A'_i$ are redefined here). Since $\bs R_i$ and $\bs R'_i$ are doubly-rooted for $i=1,2$ and rooted for $i=3$, the induction hypothesis gives similarly
	\[
		\omid{\gamma(\bs T_2,\bs T'_2)} \geq \alpha\delta\omid{\norm{A_1\cap A'_1}^{\beta'}} + \alpha\delta\omid{\norm{A_2\cap A'_2}^{\beta'}} + \delta\omid{\norm{A_3\cap A'_3}^{\beta'}}. 
	\]
	Since $n\geq N$, the definition of $N$ and~\eqref{eq:lem:induction:2} give
	\[
		\omid{\gamma(\bs T_2,\bs T'_2)} \geq \delta(1-\epsilon)n^{\beta'}\left(\alpha\omid{V_1^{\beta'}}^2 +  \alpha\omid{V_2^{\beta'}}^2 + \omid{V_3^{\beta'}}^2\right). 
	\]
	By~\eqref{eq:epsilon2}, the right hand side is at least $\alpha\delta n^{\beta'}$.
	This proves the second induction claim~\eqref{eq:thm:induction:2} and the theorem is proved.
\end{proof}

\subsection{Proof of Theorem~\ref{thm:main}}
\label{subsec:proof}

Theorem~\ref{thm:induction} is the analogous of Theorem~\ref{thm:main} for rooted and doubly-rooted trees. We will deduce Theorem~\ref{thm:main} from the following comparison lemma.

\begin{lemma}
	\label{lem:coupling}
	There exists a coupling of $(\bs T_0, \bs T'_0)$ and $(\bs T_1, \bs T'_1)$ such that $\kappa(\bs T_0,\bs T'_0)\geq\kappa(\bs T_1,\bs T'_1)$.
\end{lemma}
\begin{proof}
	By the consistency property (Lemma~\ref{lem:consistency}),
	if one removes the distinguished leaf $\star$ from $\bs T_1$, the resulting tree has the same distribution as $\bs T_0$. This provides a coupling of  $(\bs T_0, \bs T'_0)$ and $(\bs T_1, \bs T'_1)$. In this coupling, every common subtree of $\bs T_1$ and $\bs T'_1$ is also a common subtree for $\bs T_0$ and $\bs T_0$ (after removing $\star$). This proves the claim.
\end{proof}

\begin{proof}[Proof of Theorem~\ref{thm:main}]
	Subsection~\ref{subsec:alg} provides an algorithm for producing a common subtree of $\bs T_1$ and $\bs T'_1$ with size $\gamma(\bs T_1,\bs T'_1)$. So $\kappa(\bs T_1,\bs T'_1)\geq \gamma(\bs T_1,\bs T'_1)$. Now the claim is implied from Theorem~\ref{thm:induction} and Lemma~\ref{lem:coupling}.
\end{proof}

\begin{remark}
	Similarly to the above proof, one can simplify the proof of Theorem~1 of~\cite{aldous} by replacing the submartingale argument with a simple induction. In fact, the induction argument is basically using one step of the submartingale (instead of stopping it at a stopping time).
\end{remark}

\section*{Acknowledgements}
This work was supported by the ERC NEMO grant, under the European Union's Horizon 2020 research and innovation programme, grant agreement number 788851 to INRIA.

\bibliography{bib} 
\bibliographystyle{plain}

\end{document}